\documentclass[a4paper,10pt]{article}
\pdfoutput=1
\usepackage{latexsym}
\usepackage[ansinew]{inputenc}

\usepackage{amsthm}
\usepackage{amsmath}
\usepackage{amssymb}
\usepackage{amsfonts}

\usepackage{array}
\usepackage{float}
\usepackage{caption2}
\usepackage{graphicx}

\usepackage{anysize}
\marginsize{1.5cm}{1.5cm}{1.5cm}{1.5cm}

\usepackage{tabularx}

\setcaptionwidth{0.95\textwidth}%

\newcounter{remark}
\newcommand{\remark}{\addtocounter{remark}{1}\textbf{Remark
\theremark. }}

\newcommand{\beq}{\begin{equation}}
\newcommand{\enq}{\end{equation}}
\newtheorem{lemma}{Lemma}
\newtheorem{proposition}{Proposition}
\newtheorem{theorem}{Theorem}
\newtheorem{corollary}{Corollary}
\newtheorem{definition}{Definition}

\normalsize 

\title{$PSL(2,\mathbb{Z})$ as a non distorted subgroup of Thompson's group T}
\author{Ariadna Fossas Tenas\\
Institut Fourier\\
Universit\'e de Grenoble I\\
38402 St. Martin d'H\`eres. France\\
ariadna.fossas@ujf-grenoble.fr}
\date{}

\begin{document}

\maketitle

\begin{abstract}
In this paper we characterize the elements of $PSL_2(\mathbb{Z})$,
as a subgroup of Thompson group $T$, in the language of reduced
tree pair diagrams and in terms of piecewise linear maps as well.
Actually, we construct the reduced tree pair diagram for every
element of $PSL_2(\mathbb{Z})$ in normal form. This allows us to
estimate the length of the elements of $PSL_2(\mathbb{Z})$ through
the number of carets of their reduced tree pair diagrams and, as a
consequence, to prove that $PSL_2(\mathbb{Z})$ is a non distorted
subgroup of $T$. In particular, we find non-distorted free non
abelian subgroups of $T$.
\end{abstract}

\textbf{MSC classsification:} 20F65.


\section{Introduction}

Thompson's group $T$ was one of the first examples of finitely
presented infinite simple groups. There are at least three
different ways of representing the elements of Thompson's group
$T$ (see \cite{CFP} for a detailed introduction to Thompson's
groups). Probably the most common interpretation is as a subgroup
of the group of homeomorphisms of the circle, thought as the unit
interval with identified endpoints. Then, $T$ is the group of
orientation preserving piecewise linear homeomorphisms of the
circle that are differentiable except at finitely many dyadic
rational numbers, and such that, on intervals of
differentiability, the derivatives are powers of 2 (see, for
example, \cite{sergiescuGhys}).

Another interesting approach to Thompson's group $T$ describes it
as equivalence classes of tree pair diagrams. A tree pair diagram
is a pair of finite rooted binary trees with the same number of
leaves, with a cyclic numbering system pairing the leaves in the
two trees. We say that two different tree pair diagrams are
equivalent if they have the same reduced representant. This
combinatorial version of Thompson's group $T$ allowed Burillo,
Cleary, Stein and Taback to discuss metric properties of $T$ (see
\cite{BCST}). In particular, they found an estimation of the word
length in terms of the number of carets in a reduced tree pair
diagram, and showed that some subgroups of $T$ are undistorted.

We need to mention a third viewpoint about Thompson's group $T$.
We can represent $T$ as a subgroup of the group of piecewise
projective orientation preserving homeomorphisms of the real
projective line $\mathcal{R}P^1$. This result was claimed
separately by Thurston (see \cite{CFP}) and Kontsevich, and then
proved by Imbert (see \cite{imbert}) and Sergiescu (see
\cite{sergiescu}). It turns out that $T$ is isomorphic to
$PPSL_2(\mathbb{Z})$, which is the group of orientation preserving
homeomorphisms of the real projective line which are piecewise
$PSL_2(\mathbb{Z})$, and have a finite numbers of non
differentiable points, all of them being rational numbers. This
geometric approach is also related to the fact that one can
actually see Thompson's group $T$ as the asymptotic mapping class
group of an infinite surface of genus zero (see
\cite{funarKapoudjian}).

The piecewise projective approach to Thompson's group $T$ invites
us to study the projective special linear group
$PSL_2(\mathbb{Z})$ as a subgroup of $T$ using the combinatorial
methods of tree pair diagrams. It turns out that all the reduced
tree pair diagrams representing elements of $PSL_2(\mathbb{Z})$ as
a subgroup of Thompson's group $T$ have exactly one leaf and one
interior vertex for each level. In fact, we provide an explicit
bijection between tree pair diagrams of elements in
$PSL_2(\mathbb{Z})$ and the corresponding reduced words on the
classical generators of $PSL_2(\mathbb{Z})$ of order 2 and 3 (see
\cite{humphreys}). This characterization of $PSL_2(\mathbb{Z})$ in
terms of tree pair diagrams together with the results of Burillo,
Cleary, Stein and Taback in the word metric on $T$ (see
\cite{BCST}) allow us to prove that $PSL_2(\mathbb{Z})$ is a non
distorted subgroup of Thompson's group $T$.

Although it was already known that $T$ has subgroups isomorphic to
the free non abelian group of rank 2 (this is the usual way to
proof the non amenability of $T$), it was not known if these
subgroups are distorted. The piecewise projective approach
together with our results give an easy example of a non distorted
subgroup of $T$ isomorphic to $F_2$. There are several normal
subgroups of $PSL_2(\mathbb{Z})$ isomorphic to the free non
abelian group of rank two. One particular example is the
commutator group of $PSL_2(\mathbb{Z})$ (see \cite{humphreys}). In
addition, this subgroup has finite index on $PSL_2(\mathbb{Z})$.
Now, using the fact that finite index subgroups are non distorted
and our result that $PSL_2(\mathbb{Z})$ is non distorted as a
subgroup of Thompson's group $T$, we have constructed a non
distorted subgroup of Thompson's group $T$ isomorphic to the free
non abelian group of rank 2.

\subsection{Definitions and statements of the results}

First we state the definitions related to the three different ways
of representing Thompson's group $T$ introduced above:

\begin{definition}
The {\rm{piecewise linear Thompson's group $T$}} is the set of
orientation preserving piecewise linear homeomorphisms of the
circle $S^1=[0, 1]/_{0\sim1}$ that are differentiable except at
finitely many dyadic rational numbers and such that, on intervals
of differentiability, the derivatives are powers of 2.
\end{definition}

For the combinatorial definition we need the following notion:

\begin{definition} A {\rm{tree pair diagram}} is a triple
$(T_1,\sigma,T_2)$, where $T_1$ and $T_2$ are finite rooted binary
trees with the same number $n$ of leaves, and $\sigma$ is a cyclic
permutation of the set $\{1,\ldots, n\}$. The binary tree $T_1$ is
called the {\rm{source tree}}, and $T_2$ is called the {\rm{target
tree}}. A node of an ordered rooted binary tree together with its
two downward directed edges is called a {\rm{caret}}. A caret is
called {\rm{exposed}} if it contains two leaves of the tree. A
tree pair diagram is {\rm{reduced}} if, for all exposed carets of
$T_1$, the images under $\sigma$ of their leaves do not form an
exposed caret of $T_2$.
\end{definition}

When elements of Thompson's group $T$ are represented as
equivalence classes of tree pair diagrams, the multiplication of
$(T_1,\sigma_1, T_2)$ by $(T_3,\sigma_2,T_4)$ is described by the
following procedure.
\begin{enumerate}
\item Let $T_{23}$ be an ordered rooted binary tree which is a
common expansion of $T_2$ and $T_3$. This element always exists.
\item Let $(T_1',\sigma_1',T_{23})$ and $(T_{23},\sigma_2',T_4')$
be the tree pair diagrams which are equivalent to $(T_1,\sigma_1,
T_2)$ and $(T_3,\sigma_2,T_4)$, respectively. \item Then, one sets
$(T_1,\sigma_1, T_2) * (T_3,\sigma_2,T_4) =(T_1',
\sigma_1'\sigma_2', T_4')$, and reduces the obtained tree pair
diagram, if possible.
\end{enumerate}

\begin{definition}
  The {\rm{combinatorial Thompson's group $T$}} is the set of
  equivalence classes of tree pair diagrams.
\end{definition}

The third viewpoint of Thompson's group $T$ connects it with the
projective special linear group $PSL_2(\mathbb{Z})$ (see, for
example, \cite{martin}).

\begin{definition} The {\rm{piecewise projective Thompson's group $T$}}
($PPSL_2(\mathbb{Z})$) is the group of orientation preserving
homeomorphisms of the real projective line $\mathbb{R}P^1$ which
are piecewise $PSL_2(\mathbb{Z})$ and have a finite number of non
differentiable points, all of them being rational numbers.
\end{definition}

Now we need a few properties of the projective special linear
group. The group $PSL_2(\mathbb{Z})$ is isomorphic to the free
product $\mathbb{Z}/2\mathbb{Z}\ast\mathbb{Z}/3\mathbb{Z}$, which
means that it admits a generating set $\{a,b\}$, where $a^2=1$ and
$b^3=1$ (see \cite{arbres}, page 20 example 1.5.3). Then, every
element of $PSL_2(\mathbb{Z})$ has a normal form in the generators
$a$ and $b$ given by a word of the form
$a^{\epsilon_1}b^{\delta_1}ab^{\delta_2}a \ldots
ab^{\delta_k}a^{\epsilon_2}$, where $k$ is a non negative integer,
$\epsilon_1,\epsilon_2 \in \{0,1\}$ and
$\delta_1,\ldots,\delta_k\in \{-1,1\}$. We will denote by $a$ and
$b$ both the generators of $PSL_2(\mathbb{Z})$ as a group of
matrices and their images in $T$ under the isomorphism between
$PPSL_2(\mathbb{Z})$ and Thompson's group $T$.

In this paper we give a characterization of the elements of
$PSL_2(\mathbb{Z}) \subset T$ in terms of reduced tree pair
diagrams. Before stating the theorem we need the following
definition:

\begin{definition} A finite rooted binary tree is called {\rm{thin}} if
all its carets have one leave and one internal vertex except for
the last caret, which is exposed.
\end{definition}

\remark It is easy to see that one can characterize thin trees
with $n$ leaves by associating a weight $r_i \in \{-1,1\}$ to each
one of its $n-2$ internal vertices. Let $v_0, \ldots, v_{n-3}$ be
the internal vertices of a thin tree. Denote by $v_{-1}$ its root.
Then, $r_i=1$ if $v_i$ is the left descendant of $v_{i-1}$ and
$r_i=-1$ if $v_i$ is the right descendant of $v_{i-1}$.

The main result of this note is the following:

\begin{theorem}
  Let $(T_1,\sigma,T_2)$ be the reduced tree pair diagram of an element $f$ of Thompson's group $T$.
  Then, $f$ belongs to the subgroup $PSL_2(\mathbb{Z})$ if and only if $(T_1,\sigma,T_2)$
  satisfies one of the following conditions:
  \begin{enumerate}
    \item the number of leaves of $T_1$ and $T_2$ is less than 4; or
    \item the trees $T_1$ and $T_2$ are thin with associated weights $r_0, \ldots, r_{k-1}$ and $s_0, \ldots,
    s_{k-1}$ respectively, which verify the equations:
    \beq \sum_{i=2}^{k-1}r_i s_{k+1-i}= 2 - k, \label{eq1} \enq
    and
    \beq l + \sigma(1) + \epsilon(s_0) \equiv  \frac{3 - s_1}{2} \quad
    ({\rm{mod}} \, k+2),\label{eq2}
    \enq
    where $l-1$ is the cardinal of the set $\{i \, : \, r_i=-1, 0 \leq i \leq k-1\}$, and
    $\epsilon(x) = \left\{ \begin{array}{cl}
    1, & \text{\rm{if} } x = 1,\\
    0, & \text{\rm{if} } x = -1.
    \end{array} \right.$
  \end{enumerate} \label{mainThm}
\end{theorem}

This answers a question posed by Vlad Sergiescu.

In fact, given an element of $PSL_2(\mathbb{Z})$ as a subgroup of
Thompson's group $T$ in its normal form on the generating set
$\{a,b\}$, we describe an easy way to construct its reduced tree
pair diagram.

\begin{proposition}
  Let $w(a,b) = a^{\epsilon_1}b^{\delta_1}a b^{\delta_2}a \ldots a b^{\delta_k}a^{\epsilon_2}$
  be a reduced word in the standard generators $\{a,b\}$ of $PSL_2(\mathbb{Z})$, viewed as a subgroup of Thompson's group
  $T$, i.e. $\epsilon_1,\epsilon_2 \in \{0,1\}$ and $\delta_1,\ldots,\delta_k\in \{-1,1\}$. Assume that $k \geq 2$.
  Let $T_1$ be the thin tree given by the weights $r_0 = \epsilon^{-1}(\epsilon_1)$ and $r_i = \delta_i$ for $1 \leq
  i \leq k-1$, and let $T_2$ be the thin tree given by the weights $s_0 = \epsilon^{-1}(\epsilon_2)$ and $s_i = -\delta_{k+1-i}$ for
  $1 \leq i \leq k-1$. Let $\sigma$ be the cyclic permutation
  defined by
  \beq \sigma(1) \equiv \frac{3-s_1}{2} - \epsilon(s_0) - l \quad ({\rm{mod}} \, k+2),\label{eq3} \enq
  where  $l$ and $\epsilon$ are defined as in theorem 1.
  Then, $(T_1, \sigma, T_2)$ is the reduced tree pair diagram for $w$.\\
  Furthermore, the weights $r_0, \ldots, r_{k-1}$ and $s_0,\ldots, s_{k-1}$
  satisfy the equations of theorem \ref{mainThm}. \label{treesWord}
\end{proposition}

This proposition proves the `if' part of the main theorem. For the
`only if' part it suffices to show that the number of solutions to
the equations (\ref{eq1}) and (\ref{eq2}) coincides with the
number of words $w(a,b) = a^{\epsilon_1}b^{\delta_1}a
b^{\delta_2}a \ldots a b^{\delta_k}a^{\epsilon_2}$, with
$\epsilon_1,\epsilon_2 \in \{0,1\}$ and
$\delta_1,\ldots,\delta_k\in \{-1,1\}$.

Now, we can ask whether or not $PSL_2(\mathbb{Z})$ is distorted as
a subgroup of Thompson's group $T$. Recall that a finitely
generated subgroup $H$ of a finitely generated group $G$ is
distorted if there exists an infinite family $\{h_n\}_{n \in
\mathbb{N}}$ of elements of $H$ such that
\begin{enumerate}
 \item $\displaystyle |h_n|_Z < |h_{n+1}|_Z$,
 \item $\displaystyle \lim_{n \rightarrow \infty} |h_n|_Z = \infty$, and
 \item the limit
 $ \lim_{n \rightarrow \infty} \frac{|h_n|_Z}{|h_n|_Y}$
 exists and is equal to zero,
\end{enumerate}
where $Z$ is a finite generating set for the group $G$ which
contains a finite generating set $Y$ for the subgroup $H$, and
$|.|_Z$ and $|.|_Y$ denote the word metrics on the generating sets
$Z$ and $Y$, respectively. Observe that $|.|_Z \leq |.|_Y$. See
\cite{grAsympInv} for another definition of distortion.

\remark $H$ is a non distorted subgroup of $G$ if and only if
there are constants $K > 0$ and $L$ such that
$$ \frac{1}{K}|h|_Y - L \leq |h|_Z \leq K|h|_Y + L$$
holds for every $h \in H$. \label{dist}

Using theorem \ref{mainThm} and proposition \ref{treesWord} we can
estimate the length of the elements in this subgroup. Furthermore,
Burillo, Cleary, Stein and Taback gave a general estimation of the
length of elements in $T$ in the number of carets in \cite{BCST}.
As a consequence of both results, we obtain:

\begin{proposition}
  The group $PSL_2(\mathbb{Z})$ is a non distorted subgroup of
  Thompson's group $T$. \label{noDist}
\end{proposition}

In particular, we derive:

\begin{corollary}
  Let $\{a,b\}$ be the standard generating set of $PSL_2(\mathbb{Z})$ as a subgroup of Thompson's group $T$.
  Then, the subgroup $H=\left<abab,a\bar{b}a\bar{b}\right>$ is a free non abelian group of rank 2
  and it is non distorted in $T$.
\end{corollary}

Finally, we consider the elements of $PSL_2(\mathbb{Z})$ as
piecewise linear maps and we give a characterization in terms of
the coordinates of their non differentiable points $(x_0,y_0)$,
$(x_1,y_1)$, $\ldots$, $(x_k,y_k)$. As the result is technical and
needs both notation and definitions, the reader is referred to
section 4 for the details.

\textbf{Structure of the paper.} This paper is structured in four
sections. The group $T$ is presented as the group of piecewise
$PSL_2(\mathbb{Z})$ homeomorphisms of the real projective line in
section two. The third section deals with the characterization of
the elements of $PSL_2(\mathbb{Z})$ as a subgroup of $T$, the
proof that $PSL_2(\mathbb{Z})$ is a non distorted subgroup and the
construction of a non distorted free group of rank two in $T$.
Finally, in section four we give a piecewise linear
characterization of the elements of $PSL_2(\mathbb{Z})$.

\subsection{Acknowledgements.}

The author wishes to thank Jos\'e Burillo, Louis Funar and Vlad
Sergiescu for their comments and useful discussions. This work was
supported by ``La Caixa'' and ``l'Ambassade de France en Espagne''
Postgraduated Fellowship.

\section{$PSL(2,\mathbb{Z})$ as a subgroup of Thompson's group $T$}

The projective special linear group $PSL_2(\mathbb{Z})$ is
isomorphic to the free product of the cyclic groups of order 2 and
3 (see \cite{arbres}, page 20 example 1.5.3), i.e. it can be given
by the following presentation
$$
\left<  a, b | a^2=b^3=1 \right>,
$$
where $a = \left(
  \begin{array}{cc}
  0 & -1 \\
  1 & 0
  \end{array}
  \right)$
and $ b=\left(
  \begin{array}{cc}
  0 & -1 \\
  1 & 1
  \end{array}
  \right).
$

In order to connect Thompson's group $T$ with the modular group we
need the Minkowsky question mark function (see \cite{salem} or
\cite{viader}). Let $\bar{\mathbb{Q}}=\mathbb{Q} \cup \{\infty\}$
and let $\mathbb{Q}_2$ denote the dyadic rational numbers of the
unit interval. Then, the Minkowski question mark function, $? :
\bar{\mathbb{Q}} \longrightarrow \mathbb{Q}_2$, is defined
recursively. The basic cases are
$$
?\left(\infty\right)=\frac{1}{2}, \quad ?\left(\frac{0}{1}\right)=
0, \text{ and} \quad ?\left(-\infty\right)= \frac{1}{2},
$$
where $\infty$ will be represented by the fraction $\frac{1}{0}$
and $-\infty$ by $\frac{-1}{0}$. Then, for each pair of reduced
fractions, $\frac{p}{q}$ and $\frac{r}{s}$, satisfying
$|ps-qr|=1$, one defines the Minkowski question mark of their
Farey mediant $\frac{p}{q} \oplus \frac{r}{s}:=\frac{p+r}{q+s}$ as
$$?\left(\frac{p}{q} \oplus
\frac{r}{s}\right) = \frac{1}{2} \left(?\left(\frac{p}{q}\right) +
?\left(\frac{r}{s}\right)\right).$$ The Minkowski question mark
function is clearly a bijection between $\bar{\mathbb{Q}}$ and
$\mathbb{Q}_2$. Thus, using the density of $\bar{\mathbb{Q}}$ in
the real projective line $\mathbb{R}P^1 = \mathbb{R}\cup
\{\infty\}$ and $\mathbb{Q}_2$ in $[0,1]$, respectively, the
Minkowski question
mark function can be extended to $? : \mathbb{R}P^1 \longrightarrow [0,1]$. \\
The pairs of reduced fractions, $\frac{p}{q}$ and $\frac{r}{s}$,
satisfying $|ps-qr|=1$ are called consecutive Farey numbers, and
if $\frac{p}{q}< \frac{r}{s}$, the interval
$[\frac{p}{q},\frac{r}{s}]$ is called a Farey interval. Every
rational number appears as the Farey mediant of a Farey interval.
See, for example, \cite{apostol} section 5.4 or \cite{HardyWright}
chapter 3 for more details on Farey fractions.

\begin{lemma}
  Let $h$ be an element of $PPSL_2(\mathbb{Z})$ which is $ 
  \left(
  \begin{array}{cc}
  a & b \\
  c & d
  \end{array}
  \right) \in PSL_2(\mathbb{Z})$ on the interval $[x_0, x_k]$ of the real projective line $\mathbb{R}P^1$.
  Let $x_0 < x_1 < \ldots < x_k$ be a partition of $[x_0,x_k]$ such that every interval of the partition and every interval
  of the image partition $h(x_0) < \ldots < h(x_k)$ is a Farey interval. Then,
  $$h(x_i \oplus x_{i+1}) = h(x_i) \oplus h(x_{i+1}),$$
  for all $0 \leq i < k$.\label{suma}
\end{lemma}
\begin{proof}
  Let $x_i = \frac{p}{q}$ and $x_{i+1} = \frac{r}{s}$ be real rational numbers (i.e. $q,s \neq 0$).
  Then $h(x_i) = \frac{ap + bq}{cp + dq}$ and $h(x_{i+1}) = \frac{ar + bs}{cr + ds}$. Hence,
  $$h\left(\frac{p}{q} \oplus \frac{r}{s}\right)= \frac{ap + ar + bq + bs}{cp + cr + dq +ds} =
  h\left(\frac{p}{q}\right) \oplus h\left(\frac{r}{s}\right).$$
  If $x_{i+1} = \frac{1}{0}$, then $h(x_{i+1}) = \frac{a}{c}$ and
  $$h\left(\frac{p}{q} \oplus \frac{1}{0}\right)= \frac{ap + a + bq}{cp + c + dq} =
  h\left(\frac{p}{q}\right) \oplus h\left(\frac{1}{0}\right).$$
  Analogously if $x_i = \frac{-1}{0}$.
\end{proof}

The following result, due to Imbert, Kontsevich and Sergiescu,
identifies the two groups encoutered above (see \cite{imbert},
theorem 1.1).
\begin{theorem} (\cite{imbert}, theorem 1.1)
  The group $PPSL_2(\mathbb{Z})$ is isomorphic to Thompson's group $T$. \label{imbert}
\end{theorem}
\begin{proof}
  We claim that the homomorphism $Inn_{?}: PPSL_2(\mathbb{Z}) \rightarrow T$, given by
  $Inn_{?}(g)=? \circ g \, \circ \, ?^{-1}$ is an isomorphism.
  First, we consider the generating set $\{a,b\}$ of $PSL_2(\mathbb{Z})$ and calculate $Inn_{?}(a)$ and $Inn_{?}(b)$.
  By lemma \ref{suma}, it suffices to find Farey partitions of $\bar{\mathbb{Q}}$ whose images
  are also Farey partitions. Then, applying the Minkowsi question mark function
  we will obtain partitions of the unit interval characterizing finite binary trees with the same number
  of leaves, thus elements in $T$. The following table summarizes this procedure.

  \renewcommand {\tabularxcolumn }[1]{ >{\arraybackslash }m{#1}}
\begin{center}
  \begin{tabularx}{0.85\textwidth}{|c|c|c|X|}
  \hline
  Generator & Farey partitions & Piecewise linear map & Reduced tree pair diagram \\ \hline \hline
  $\left(
  \begin{array}{cc}
  0 & -1 \\
  1 & 0
  \end{array}
  \right)$
  &
  $\begin{array}{rcl}
  & & \\
  \left[\frac{0}{1},\frac{1}{0}\right] & \mapsto & \left[\frac{1}{0},\frac{0}{1}\right]\\ \\
  \left[\frac{1}{0},\frac{0}{1}\right] & \mapsto & \left[\frac{0}{1},\frac{1}{0}\right]\\ \\
  \end{array}$
  &
  $ \left\{
  \begin{array}{cc}
  &  \\
  \displaystyle x + \frac{1}{2}, & 0 \leq x \leq \frac{1}{2}\\ \\
  \displaystyle x - \frac{1}{2}, & \frac{1}{2} \leq x \leq 1. \\ \\
  \end{array}
  \right.
  $
  &
  \includegraphics[width=3.5cm]{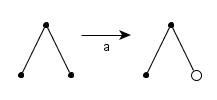} \\ \hline

  $ \left(
  \begin{array}{cc}
  0 & -1 \\
  1 & 1
  \end{array}
  \right)$
  &
  $\begin{array}{rcl}
  & & \\
  \left[\frac{0}{1},\frac{1}{0}\right] & \mapsto & \left[\frac{-1}{1},\frac{0}{1}\right] \\ \\
  \left[\frac{1}{0},\frac{-1}{1}\right] & \mapsto & \left[\frac{0}{1},\frac{1}{0}\right] \\ \\
  \left[\frac{-1}{1},\frac{0}{1}\right] & \mapsto & \left[\frac{1}{0},\frac{-1}{1}\right] \\ \\
  \end{array}$
  &
  $ \left\{
  \begin{array}{cc}
  &  \\
  \displaystyle \frac{x}{2} + \frac{3}{4}, & 0 \leq x \leq \frac{1}{2}\\ \\
  \displaystyle 2x - 1, & \frac{1}{2} \leq x \leq \frac{3}{4}\\ \\
  \displaystyle x - \frac{1}{4}, & \frac{3}{4} \leq x \leq 1. \\ \\
  \end{array}
  \right.$
  &
  \includegraphics[width=4cm]{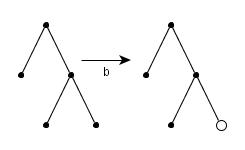} \\ \hline
  \end{tabularx}
\end{center}

  Thanks to this and lemma \ref{suma}, the map $Inn_{?}$ is well defined. Thus, $Inn_{?}$ is a homomorphism.
  The injectivity is also a consequence of lemma \ref{suma}. Let $A$, $B$ and $C$ be the three classical generators of
  Thompson's group $T$, i.e. the elements with the following reduced tree pair diagrams (see \cite{CFP} for
  details).

\begin{center}
\includegraphics[width=15cm]{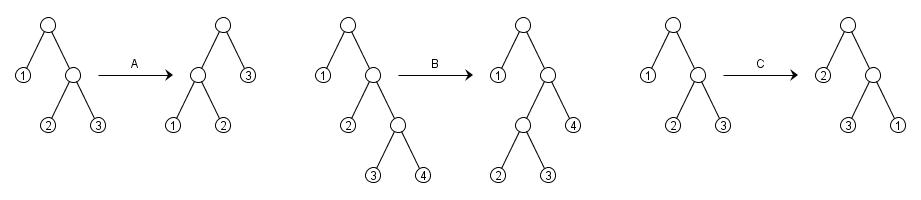}
\end{center}

  Note that $Inn_{?}(b) = C$ and $Inn_{?}(a) = CA$. In order to prove that
  $Inn_{?}$ is surjective, we have to find and element $d$ in $PPSL_2(\mathbb{Z})$ such that $Inn_{?}(d)=B$. This element is
  $$
  d(x) = \left\{
  \begin{array}{ll}
  \left(\begin{array}{cc} 1 & 0 \\ 0 & 1 \end{array}\right), & \text{if } 0 \leq x \leq \frac{1}{0},\\ \\
  \left(\begin{array}{cc} 1 & -1 \\ 0 & 1 \end{array}\right), & \text{if } \frac{-1}{0} \leq x \leq \frac{-1}{1},\\ \\
  \left(\begin{array}{cc} 3 & 1 \\ -1 & 0 \end{array}\right), & \text{if } \frac{-1}{1} \leq x \leq \frac{-1}{2},\\ \\
  \left(\begin{array}{cc} 1 & 0 \\ 1 & 1 \end{array}\right), & \text{if } \frac{-1}{2} \leq x \leq \frac{0}{1}.
  \end{array}\right.
  $$
\end{proof}

As a direct consequence of theorem \ref{imbert},
$PSL_2(\mathbb{Z})$ can be considered as a subgroup of Thompson's
group $T$. By abuse of notation, from now on we will denote the
generators in $PSL_2(\mathbb{Z})$ and its images in $T$ with the
same symbol, $a$ and $b$, and the subgroup
$Inn_{?}(PSL_2(\mathbb{Z}))$ will be simply $PSL_2(\mathbb{Z})$.

\remark  Let $A$, $B$ and $C$ be the three classical generators of
Thompson's group $T$ (see \cite{CFP} for details). It is easy to
see that $a = CA$ and $b = C$. Hence, $PSL_2(\mathbb{Z}) \subset
T$ is the subgroup generated by $A$ and $C$.

\section{Tree diagram characterization of $PSL_2(\mathbb{Z})$}

Recall that an ordered rooted binary tree is thin if all its
carets have one leave and one internal vertex, except for the last
caret. Furthermore, there exists a bijection between the set of
thin trees with $n$ leaves and $\{-1,1\}^{n-2}$, by associating a
weight $r_i \in\{-1,1\}$ to every internal vertex $v_i$, for $0
\leq i \leq n-3$, in the following way: $r_i=1$ if $v_i$ is the
left descendant of $v_{i-1}$ and $r_i=-1$ if $v_i$ is the right
descendant of $v_{i-1}$, where $v_{-1}$ denotes the root.

\remark A thin tree with $n$ leaves has exactly $n-2$ associated
weights and $n-1$ carets. Furthermore, if $l$ is the number of the
left leaf of the exposed caret, then $l-1$ is the cardinal of the
set $\{i \, : \, r_i=-1, 0 \leq i \leq k-1\}$.

Recall that we want to characterize the elements of
$PSL_2(\mathbb{Z})$ as a subgroup of Thompson's group $T$ in terms
of reduced tree pair diagrams.

\setcounter{theorem}{0}
\begin{theorem}
  Let $(T_1,\sigma,T_2)$ be the reduced tree pair diagram of an element $f$ of Thompson's group $T$.
  Then, $f$ belongs to the subgroup $PSL_2(\mathbb{Z})$ if and only if $(T_1,\sigma,T_2)$
  satisfies one of the following conditions:
  \begin{enumerate}
    \item the number of leaves of $T_1$ and $T_2$ is less than 4; or
    \item the trees $T_1$ and $T_2$ are thin with associated weights $r_0, \ldots, r_{k-1}$ and $s_0, \ldots,
    s_{k-1}$ respectively, which verify the equations:
    $$ \sum_{i=2}^{k-1}r_i s_{k+1-i}= 2 - k, $$
    and
    $$ l + \sigma(1) + \epsilon(s_0) \equiv  \frac{3 - s_1}{2} \quad ({\rm{mod}} \, k+2),
    $$
    where $l-1$ is the cardinal of the set $\{i \, : \, r_i=-1, 0 \leq i \leq k-1\}$, and
    $\epsilon(x) = \left\{ \begin{array}{cc}
    1 & \text{if } x = 1\\
    0 & \text{if } x = -1.
    \end{array} \right.$
  \end{enumerate}
\end{theorem}

In particular, there is a relation between reduced words on the
generators $a$ and $b$ of $PSL_2(\mathbb{Z})$ and reduced tree
pair diagrams representing them as elements of Thompson's group
$T$. We will denote the inverse of the generator $b$ by $b^{-1}$
or $\bar{b}$.

\setcounter{proposition}{0}
\begin{proposition}
  Let $w(a,b) = a^{\epsilon_1}b^{\delta_1}a b^{\delta_2}a \ldots a b^{\delta_k}a^{\epsilon_2}$
  be a reduced word on the generators $\{a,b\}$ of $PSL_2(\mathbb{Z})$, 
  where $\epsilon_1,\epsilon_2 \in \{0,1\}$ and $\delta_1,\ldots,\delta_k\in \{-1,1\}$. Assume that $k \geq 2$.
  Let $T_1$ be the thin tree given by the weights $r_0 = \epsilon^{-1}(\epsilon_1)$ and $r_i = \delta_i$ for $1 \leq
  i \leq k-1$, and let $T_2$ be the thin tree given by the weights $s_0 = \epsilon^{-1}(\epsilon_2)$ and $s_i = -\delta_{k+1-i}$ for
  $1 \leq i \leq k-1$. Let $\sigma$ be the cyclic permutation
  defined by
  $$\sigma(1) \equiv \frac{3-s_1}{2} - \epsilon(s_0) - l \quad ({\rm{mod}} \, k+2),$$
  where  $l$ and $\epsilon$ are defined as in theorem 1.
  Then, $(T_1, \sigma, T_2)$ is the reduced tree pair diagram for $w$.\\
  Furthermore, the weights $r_0, \ldots, r_{k-1}$ and $s_0,\ldots, s_{k-1}$
  satisfy equations (\ref{eq1}) and (\ref{eq2}) of theorem \ref{mainThm}. \label{treesWord}
\end{proposition}
\begin{proof}
  We will proceed by induction. First, we consider the eight cases with $k=2$ and $\epsilon_2=0$,
  so that the reduced word is $w=a^{\epsilon_1}b^{\delta_1}ab^{\delta_2}$.

  \begin{tabular}{|c|c|c|c|}
    \hline
    \includegraphics[width=4cm]{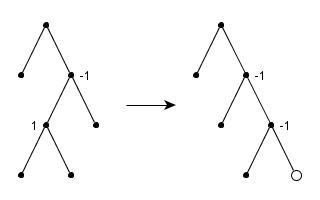} & \includegraphics[width=3.6cm]{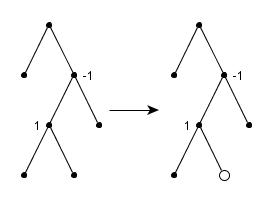} & \includegraphics[width=4cm]{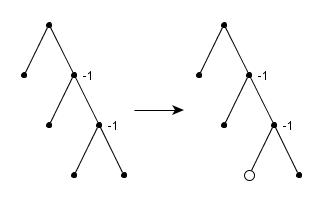} & \includegraphics[width=3.9cm]{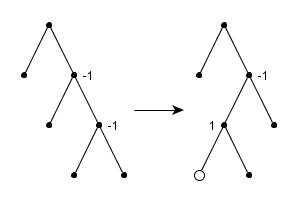}\\
    $bab$ & $ba\bar{b}$ & $\bar{b}ab$ & $\bar{b}a\bar{b}$ \\ \hline
    \includegraphics[width=4cm]{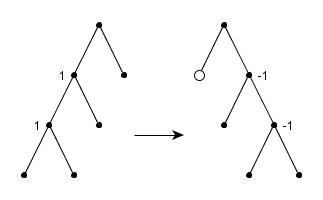} & \includegraphics[width=3.9cm]{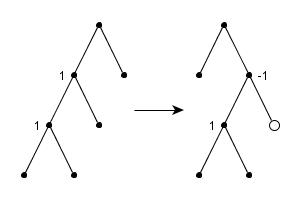} & \includegraphics[width=3.9cm]{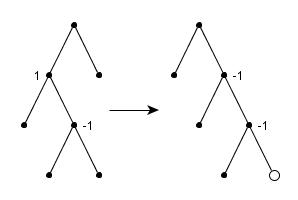} & \includegraphics[width=3.6cm]{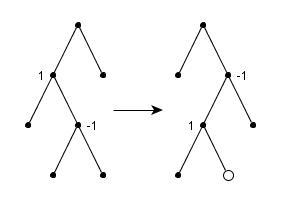}\\
    $abab$ & $aba\bar{b}$ & $a\bar{b}ab$ & $a\bar{b}a\bar{b}$\\
    \hline
  \end{tabular}

  The weights of their reduced tree pair diagrams verify the
  equations of the proposition
  $$
  \begin{array}{ccl}
  r_0 & = & \epsilon^{-1}(\epsilon_1)\\
  r_1 & = & \delta_1 \\
  s_0 & = & -1 \\
  s_1 & = & -\delta_2,
  \end{array}
  $$
  and
  $$\sigma(1) \equiv \frac{3-s_1}{2} - l \quad (\text{mod } 4).$$

  Now, let
  $w(a,b) = a^{\epsilon_1}b^{\delta_1}a b^{\delta_2}a \ldots a b^{\delta_k}a^{\epsilon_2}$
  be a reduced word verifying the induction hypothesis.
  We will see that all reduced words of type $w' = wy$, where $y \in \{a,b, \bar{b}\}$,
  also satisfy the equations of the proposition. There are several cases to be considered.

  a.- Suppose that $\epsilon_2 = 0$. Then, $w'=wa$. By the induction hypothesis, $wa$ has
  the reduced tree pair diagram of figure \ref{fig:TDW1}.
  \begin{figure}[H]
  \centering
  \includegraphics[width=8cm]{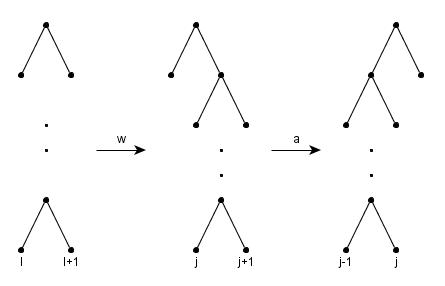}
  \caption{Tree pair diagrams multiplication for $wa$.}
  \label{fig:TDW1}
  \end{figure}
  The tree pair diagram for $wa$ in figure \ref{fig:TDW1} is reduced because the tree pair
  diagram for $w$ in the same figure was reduced. As a consequence, $r_i'=r_i$ for
  $0 \leq i \leq k-1$, $s_0'=1$, $s_i'=s_i$ for
  $1 \leq i \leq k-1$, $l' = l$ and $\sigma'(1) \equiv \sigma(1) -1 \, (\text{mod }k+2)$.
  Since $w'=wa$, we have $\epsilon_1' = \epsilon_1$, $\epsilon_2' = 1$ and
  $\delta_i' = \delta_i$ for $1 \leq i \leq k$. Thus, equation (\ref{eq1}) for $w'$
  follows from equation (\ref{eq1}) for $w$. Further,
  $l'+ \sigma'(1) + \epsilon(\epsilon_2')= l + \sigma(1) - 1+ 1$, which is congruent to
  $\frac{3 -  s_1'}{2}$ by the induction hypothesis, thereby proving that equation (\ref{eq2})
  is also verified.

  b.- Suppose that $\epsilon_2 = 1$. Then, by the induction hypothesis $w$ has the reduced
  tree pair diagram $(T_1,\sigma,T_2)$ in figure \ref{fig:TDW2}.
  \begin{figure}[H]
  \centering
  \includegraphics[width=5cm]{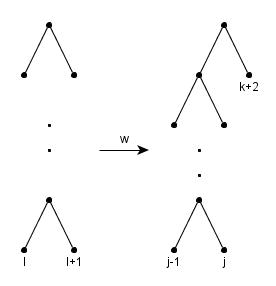}
  \caption{Reduced tree pair diagram for $w$.}
  \label{fig:TDW2}
  \end{figure}
  Note that in order to apply $b$ or $\bar{b}$ to $T_2$ we need to split the
  $(k+2)$-th vertex of $T_2$. That makes us split the $p$-th vertex of $T_1$ as well, where $p$ satisfies $\sigma(p) = k+2$.
  Since $\sigma(p) = \sigma(1) + p - 1$, then $\sigma(1) + p - 1 = k+2$, i.e.
  $p \equiv k + 3 - \sigma(1) \equiv 1 - \sigma(1) \, (\text{mod }k+2)$.
  We want to see that $p \in \{l, l+1\}$. By the induction hypothesis we have
  $\sigma(1) \equiv 1 - l$ if $\delta_k = 1$ and $\sigma(1) \equiv - l$ if $\delta_k = -1$.
  Thus, $p \equiv l$ if $\delta_k = 1$ and $p \equiv l+1$ if $\delta_k = -1$.

  Now, we must consider four subcases:

  \begin{tabular}{|l|l|}
    \hline
    & \\
    b1.- $\delta_k=1$, $w'=wb$ ($\delta_{k+1}=1$) & b2.- $\delta_k=1$, $w'=w\bar{b}$ ($\delta_{k+1}=-1$) \\
    \includegraphics[width=8cm]{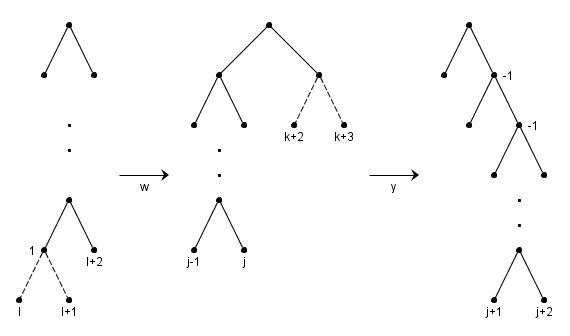} & \includegraphics[width=7.8cm]{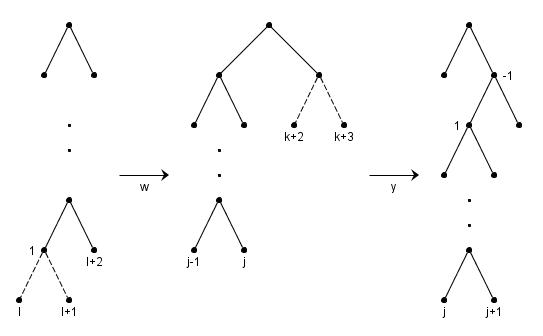} \\ \hline
    & \\
    b3.- $\delta_k=-1$, $w'=wb$ ($\delta_{k+1}=1$) & b4.- $\delta_k=-1$, $w'=w\bar{b}$ ($\delta_{k+1}=-1$) \\
    \includegraphics[width=8cm]{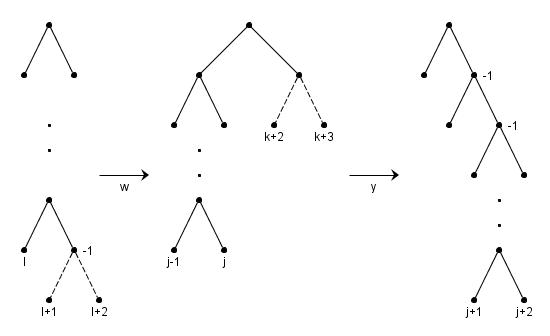} & \includegraphics[width=7.8cm]{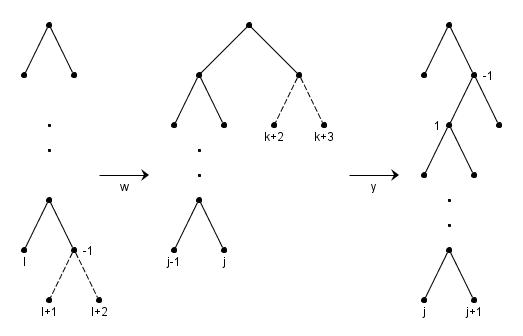} \\ \hline
  \end{tabular}

  It is clear that the tree pair diagrams obtained for $w'$ are reduced because the tree pair diagram for $w$ was
  reduced. Note also that all four subcases satisfy the following relations:
  $$
  \begin{array}{cccl}
  (r1) & k'   & = & k+1, \\
  (r2) & r_i' & = & r_i, \quad \text{for } 0 \leq i \leq k-1,\\
  (r3) & r_k' & = & \delta_k, \\
  (r4) & s_0' & = & -1, \\
  (r5) & s_1' & = & -\delta_{k+1}, \text{ and } \\
  (r6) & s_i' & = & s_{i-1}, \quad \text{for } 2 \leq i \leq k.\\
  \end{array}
  $$
  As a consequence of $(r1)$, $(r2)$ and $(r6)$ we obtain
  $$
  \sum_{i=2}^{k'-1}r_i's_{k'+1-i} = \sum_{i=2}^{k-1}r_i s_{k+1-i} + r_k's_2'.
  $$
  Thus, applying the induction hypothesis we get
  $$ \sum_{i=2}^{k'-1}r_i's_{k'+1-i} = 2 - k - r_k' \delta_k.$$
  Therefore, equation (\ref{eq1}) holds if and only if $r_k'\delta_k = 1$, and this follows from $(r3)$.

  Finally, we must verify that our four subcases satisfy equation (\ref{eq2}). Note that $(r4)$ implies
  $\epsilon(s_0')=0$, so the left-hand-side of the equation (\ref{eq2}) for $w'$ is equal to $l' + \sigma'(1)$.
  The reduced tree pair diagrams show that
  $$
  l' + \sigma'(1) = \left\{
  \begin{array}{cl}
  l + \sigma(1) + 2, & \text{in case b1} \\
  l + \sigma(1) + 1, & \text{in case b2} \\
  l + \sigma(1) + 3, & \text{in case b3} \\
  l + \sigma(1) + 2, & \text{in case b4.} \\
  \end{array}
  \right.
  $$
  On the other hand, since $w$ verifies equation (\ref{eq2}) we
  obtain
  $$
  l + \sigma(1) + 1 \equiv \left\{
  \begin{array}{ccl}
  1, & \text{if }\delta_k = -1 & \text{(cases b1 and b2)}\\
  2, & \text{if }\delta_k = 1 & \text{(cases b3 and b4)}
  \end{array}\right. \, (\text{mod } k+2).
  $$
  Furthermore, $1 \leq l \leq k+1$ and $1 \leq \sigma(1) \leq k+2$,
  so
  $$
  l + \sigma(1) + 1 = \left\{
  \begin{array}{ccl}
  k + 3, & \text{if }\delta_k = -1 & \text{(cases b1 and b2)}\\
  k + 4, & \text{if }\delta_k = 1 & \text{(cases b3 and b4)}.
  \end{array}\right.
  $$
  Hence, equation (\ref{eq2}) holds for $w'$.
\end{proof}

\begin{proof} (of theorem \ref{mainThm}).
First, we list all reduced tree pair diagrams with less than 4
leaves and their reduced words in the generators $a$ and $b$ of
$PSL_2(\mathbb{Z})$ (see the table below). Then, we focus on tree
pair diagrams with thin trees and more than 3 leaves. The `if'
part coincides with proposition \ref{treesWord}. The `only if'
part will be proved by counting the number of elements satisfying
the two equations and showing this is exactly the number of
elements that come from reduced words in $a$ and $b$.

\begin{center}
\begin{tabular}{|c|c|c|c|c|}
\hline
\includegraphics[width=2cm]{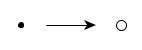} & \includegraphics[width=2.5cm]{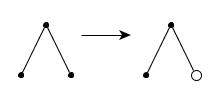} & \includegraphics[width=3cm]{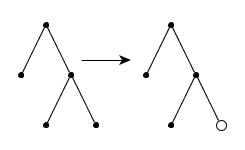} &
\includegraphics[width=3cm]{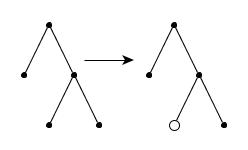} & \includegraphics[width=3cm]{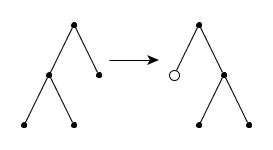} \\
1 & $a$ & $b$ & $\bar{b}$ & $ab$ \\ \hline
\includegraphics[width=3cm]{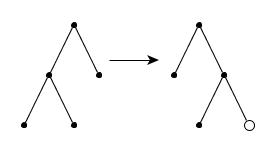} & \includegraphics[width=3cm]{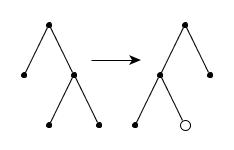} & \includegraphics[width=3cm]{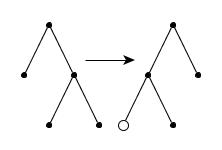} &
\includegraphics[width=3cm]{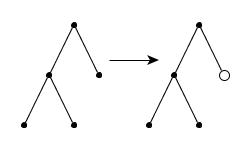} & \includegraphics[width=3cm]{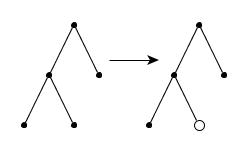} \\
$a\bar{b}$ & $ba$ & $\bar{b}a$ & $aba$ & $a\bar{b}a$ \\ \hline
\end{tabular}
\end{center}

Given a thin source tree, equation (\ref{eq1}) implies that we can
only have four different thin target trees, namely those with $s_0
\in \{0,1\}$ and $s_1 \in \{-1,1\}$. Moreover, given a pair of
compatible thin trees, $l$ is determined by the source tree and
all other terms of equation (\ref{eq2}) except $\sigma(1)$ are
determined by the target tree. Then, there exists a unique
permutation $\sigma$ satisfying equation (\ref{eq2}). Therefore,
given a thin tree $T_1$ there exist exactly four reduced tree pair
diagrams having $T_1$ as the source tree and satisfying equations
(\ref{eq1}) and (\ref{eq2}), which is exactly the same number of
reduced word in $a, \, b$ whose associated source tree is $T_1$.
\end{proof}

\remark From theorem \ref{mainThm} and proposition
\ref{treesWord}, we can read the reduced word in generators $a, b$
of any element of $PSL_2(\mathbb{Z})$ directly from its reduced
tree pair diagram.

Next, we wonder whether $PSL_2(\mathbb{Z})$ is distorted as a
subgroup of Thompson's group $T$.

\begin{proposition}
  The group $PSL_2(\mathbb{Z})$ is a non distorted subgroup of Thompson's group $T$. \label{noDist}
\end{proposition}

For the proof, we will need the following theorem due to Burillo,
Cleary, Stein and Taback.

\begin{theorem} (\cite{BCST}, theorem 5.1)
  Let $N(w)$ be the number of carets of the reduced tree pair diagram representing
  $w \in T$, and let $|w|_{A,B,C}$ denote the word length of $w$ in the classical
  generating set $\{A,B,C\}$. Then, there exists a positive constant $K$ such that
  $$\frac{N(w)}{K} \leq |w|_{A,B,C} \leq K N(w).$$ \label{thmBCST}
\end{theorem}

\begin{proof} (of proposition \ref{noDist})
Let $w=a^{\epsilon_1}b^{\delta_1}a \ldots a
b^{\delta_k}a^{\epsilon_2}$ be a reduced word in the generators
$a$ and $b$ of $PSL_2(\mathbb{Z})$, where $\delta_i \in \{-1,1\}$
and $\epsilon_j \in \{0,1\}$. We want to find constants $K'>0$ and
$L$ such that
$\frac{|w|_{a,b}}{K'}-L \leq |w|_{A,B,C,a,b} \leq K'|w|_{a,b}+L$ (see remark \ref{dist}).\\
Let $N(w)$ be the number of carets of the reduced tree pair
diagram representing $w \in T$. Recall that $a = CA$ and $b= C$.
Thus,
$$
\frac{|w|_{A,B,C}}{2} \leq |w|_{A,B,C,a,b} \leq |w|_{A,B,C}.
$$
Then, applying theorem \ref{thmBCST} we obtain \beq
\frac{N(w)}{2K} \leq |w|_{A,B,C,a,b} \leq 2K N(w). \label{lenghtT}
\enq On the other hand,
$$|w|_{a,b}= |a^{\epsilon_1}b^{\delta_1}a \ldots a b^{\delta_k}a^{\epsilon_2}|_{a,b}= 2k - 1 + \epsilon_1 + \epsilon_2.$$
From proposition \ref{treesWord} and theorem \ref{mainThm}, the
reduced tree pair diagram of $w$ has exactly $k+2$ leaves and
$k+1$ carets. Thus,
$$ 2 N(w) - 3 \leq |w|_{a,b} \leq 2 N(w) - 1,$$
which implies that \beq \frac{|w|_{a,b}+1}{2} \leq N(w) \leq
\frac{|w|_{a,b}+3}{2}. \label{N} \enq Therefore, from the
equations (\ref{lenghtT}) and (\ref{N}) we obtain the inequalities
$$
\frac{|w|_{a,b}+1}{4K} \leq |w|_{A,B,C,a,b} \leq K(|w|_{a,b} + 3),
$$
which yield
$$
\frac{|w|_{a,b}}{K'} - L \leq |w|_{A,B,C,a,b} \leq K'|w|_{a,b} +
L,
$$
where $K' = 4K$ and $L = 3K$.
\end{proof}

\setcounter{corollary}{0}
\begin{corollary}
  There exists a non distorted subgroup of Thompson's group $T$ isomorphic to
  the free non abelian group of rank 2.
\end{corollary}
\begin{proof}
  Let $H$ be the subgroup of $PSL_2(\mathbb{Z})$ generated by $g=abab$ and
  $h=a\bar{b}a\bar{b}$, which
  is isomorphic to the free non abelian group of rank 2 (see
  \cite{PdelaHarpe}, page 26 or \cite{Lyndon-Shupp} proposition
  III.12.3). The subgroup $H$ is the kernel of the morphism
  $$
  \begin{array}{ccc}
    PSL_2(\mathbb{Z}) &
    \longrightarrow & \mathbb{Z}/2\mathbb{Z} \times
    \mathbb{Z}/3\mathbb{Z}\\
    a & \longmapsto & (1,0)\\
    b & \longmapsto & (0,1),
  \end{array}$$
  so $H$ is of finite index in $PSL_2(Z)$. Thus, $H$ is non
  distorted in $PSL_2(\mathbb{Z})$ and using the proposition \ref{noDist}
  $H$ is non distorted in $T$.
\end{proof}

This is related to recent results of Calegari and Freedman
\cite{calegariFreedman}, who analyzed distorted cyclic subgroups
of more general homeomorphisms groups.

\section{Piecewise linear characterization of $PSL_2(\mathbb{Z})$}

In this section we will characterize the elements of
$PSL_2(\mathbb{Z})$ as piecewise linear maps of the unit interval
with identified endpoints. Note that a piecewise linear map of
$[0,1]$ with identified endpoints $f$ can be given by the
coordinates of its non differentiable points $(x_0,y_0)$,
$(x_1,y_1)$, $\ldots$, $(x_k,y_k)$. Furthermore, $x_0=0$, $x_k=1$
and there exists $i \in \{0, \ldots, k\}$ such that $y_i=0$. Let
$\Delta_j(x)$ denote the difference $x_j-x_{j-1}$ for $1 \leq j
\leq k$, and $\Delta_j(y)$ denote $y_j-y_{j-1}$. Thus, we can
represent the map $f$ by the sequence
$$S(f)=\left\{\frac{\Delta_1(y)}{\Delta_1(x)},
\frac{\Delta_2(y)}{\Delta_2(x)}, \ldots,
\frac{\Delta_i(y)}{\Delta_i(x)}, \frac{\circ
\Delta_{i+1}(y)}{\Delta_{i+1}(x)}, \ldots,
\frac{\Delta_k(y)}{\Delta_k(x)} \right\},$$ since, for $0 \leq
j\leq k$ we have
$$x_j =\sum_{l=1}^j \Delta_l \quad \text{and} \quad y_j
=\sum_{l=i+1}^{i+1+j}\Delta_{l(\text{mod } k)}.$$

Note that the mark $\circ$ on the sequence $S(f)$ denotes the
cyclical permutation which prevents us to have $y_0=0$. We will
denote by $\Delta_x(f)$ the sequence $\Delta_1(x)$, $\Delta_2(x)$,
$\ldots$, $\Delta_k(x)$, by $\Delta_y(f)$ the sequence
$\Delta_1(y)$, $\Delta_2(y)$, $\ldots$, $\Delta_k(y)$ and by
$\Delta_y^{\circ}(f)$ the sequence $\Delta_{i+1}(y)$,
$\Delta_{i+2}(y)$, $\ldots$, $\Delta_k(y)$, $\Delta_1(y)$,
$\ldots$, $\Delta_i(y)$. We will denote by $\Delta(f)$ either
$\Delta_x(f)$, $\Delta_y(f)$ or $\Delta_y^{\circ}(f)$.

\remark A piecewise linear map $f$ is an element of Thompson group
$T$ if and only if for all $j \in \{1, \ldots,k\}$, $x_j$, $y_j$
are dyadic rational numbers and the fraction
$\frac{\Delta_j(y)}{\Delta_j(x)}$ is a power of two.

\begin{definition} Let $\sigma$ be a permutation of the set $\{1,
\ldots, k\}$. Then, $\sigma$ is $k-${\rm{extremal}} if:
\begin{enumerate}
\item $\sigma(1)=1$ or $\sigma(1)=k$; \item for $2 \leq i \leq
k-2$, $\sigma(i)$ is the maximum or the minimum of the set $\{1,
\ldots, k \}\setminus \bigcup_{j=1}^{i-1}\sigma(j)$; \item
$\sigma(k-1)$ is the minimum and $\sigma(k)$ is the maximum of the
pair $\{1, \ldots, k \}\setminus \bigcup_{j=1}^{k-2}\sigma(j)$.
\end{enumerate}
\end{definition}

\textbf{Example:} The permutation $\sigma = (1,2,3,4,5) \mapsto
(5,1,2,3,4)$ is $5-$extremal.

\begin{definition} The sequence $\Delta$ is
$k-${\rm{thin}} if there exists $\sigma$ which is $k-$extremal
such that:
\begin{description}
  \item[(t1)] $\Delta_{\sigma^{-1}(i)}=2^{-i}$, for $1 \leq i \leq k-1$, and
  \item[(t2)] $\Delta_{\sigma^{-1}(k)}=2^{-k+1}$.
\end{description}
\end{definition}

\begin{lemma}
There exists a canonical bijection between $k-$thin sequences and
thin trees with $k$ leaves. \label{lemma}
\end{lemma}
\begin{proof}
  Let $\Delta$ be a $k-$thin sequence. Note that $\Delta=\left\{\Delta^{(-1)}, 2^{-k+1} , 2^{-k+1}
  , \Delta^{(1)}\right\}$, where $\Delta^{(-1)}$ and $\Delta^{(1)}$ are subsequences of $\Delta$.
  For $0 \leq i \leq k-3$ we define $r_i=j$ if $\Delta_{\sigma^{-1}(i+1)} \in \Delta^{(j)}$.
  Then, we associate the thin tree with weights $r_0, \ldots, r_{k-3}$ to $\Delta$.
  Finally, we remark that the number of thin trees with $k$ leaves and the number
  of $k-$thin permutations are both equal to $2^{k-2}$.
\end{proof}

\textbf{Example:} A $5-$thin sequence and its associated thin tree
are shown below.

\vspace{-1cm}
\begin{tabular}{m{3cm}m{1cm}}
$$\begin{array}{lcl}
  \Delta & = & \left\{2^{-2},2^{-4},2^{-4},2^{-3},2^{-1}\right\}.\\
  & & \\
  \sigma & = & (1,2,3,4,5)\mapsto(5,1,4,2,3).\\
  & & \\
  \Delta^{(-1)} & = & \left\{2^{-2}\right\} \quad \Rightarrow \quad r_1=-1.\\
  & & \\
  \Delta^{(1)} & = & \left\{2^{-3},2^{-1}\right\} \quad \Rightarrow \quad r_2=r_0=1.
  \end{array}$$
&
\begin{figure}[H]
\begin{center}
\includegraphics[width=3cm]{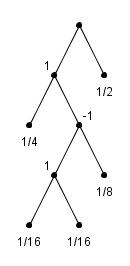}
\end{center}
\end{figure}
\end{tabular}

\begin{definition} The sequence $S(f)$ representing a piecewise
linear map $f$ is $k-${\rm{good}} if
\begin{description}
  \item[(g1)] the sequence $\Delta_x(f)$ is $k-$thin,
  \item[(g2)] the sequence $\Delta_y^{\circ}(f)$ is $k-$thin, and
  \item[(g3)] for $3 \leq j \leq k-2$, the fraction $\frac{2^{-k-1+j}}{2^{-j}}$ appears in $S(f)$.
\end{description}
\end{definition}

\textbf{Example:} The sequence $\displaystyle S(f)=
\left\{\frac{2^{-5}}{2^{-3}},\frac{2^{-3}}{2^{-5}},\frac{2^{-2}}{2^{-6}},
\frac{\circ
2^{-1}}{2^{-6}},\frac{2^{-4}}{2^{-4}},\frac{2^{-6}}{2^{-2}},\frac{2^{-6}}{2^{-1}}\right\}$
is $7-$good.

\begin{theorem}
An element $f$ of $T$ belongs to $PSL_2(\mathbb{Z})$ if and only
if the sequence $S(f)$ is $k-$good.
\end{theorem}
\begin{proof}
  For the `if' part, we first enumerate the $2-$good and $3-$good sequences $S(f)$ with
  their corresponding elements of $PSL_2(\mathbb{Z})$ in the following table.

  \begin{tabular}{|c|c|c|c|c|c|c|c|c|c|}
  \hline & & & & & & & & & \\
    $\displaystyle \frac{\circ 1}{1},\frac{1}{1}$ & $\displaystyle \frac{1}{1},\frac{\circ 1}{1}$
    & $\displaystyle \frac{\circ 1}{2},\frac{1}{1},\frac{2}{1}$ & $\displaystyle \frac{1}{2},\frac{1}{1},\frac{\circ 2}{1}$
    & $\displaystyle \frac{1}{2},\frac{2}{1},\frac{\circ 1}{1}$ & $\displaystyle \frac{1}{2},\frac{\circ 2}{1},\frac{1}{1}$
    & $\displaystyle \frac{2}{1},\frac{\circ 1}{1},\frac{1}{2}$ & $\displaystyle \frac{\circ 2}{1},\frac{1}{1},\frac{1}{2}$
    & $\displaystyle \frac{1}{1},\frac{2}{1},\frac{\circ 1}{2}$ & $\displaystyle \frac{1}{1},\frac{\circ 2}{1},\frac{1}{2}$
    \\ & & & & & & & & & \\
    1 & $a$ & $\bar{b}a$ & $\bar{b}$ & $ba$ & $b$ & $aba$ & $ab$ & $a\bar{b}a$ & $a\bar{b}$ \\ \hline
  \end{tabular}

  Now suppose that we are given a $k-$good sequence $S(f)$ with
  $k\geq 4$. Let $r_0,\ldots, r_{k-3}$ be the weights associated to
  $\Delta_x(f)$ and $s_0,\ldots,s_{k-3}$ be the weights associated to
  $\Delta_y^{\circ}(f)$ using the lemma \ref{lemma}.
  First, we will prove that $r_0, \ldots, r_{k-3}$ and $s_0,
  \ldots, s_{k-3}$ satisfy the equation \ref{eq1} of the theorem
  \ref{mainThm}, by showing that $r_j = -s_{k-1-j}$ for $2 \leq j \leq k-3$.

  From (g3) we know that, for $3 \leq j \leq k-2$, $S(f)$ contains the term $\frac{2^{-k+j-1}}{2^{-j}}$.
  Let $2 \leq i,j\leq k-3$. Thus, we have one of the following situations:
  \begin{enumerate}
  \item $S(f)=\displaystyle \left\{\frac{2^{-k+1}}{2^{-1}},\frac{2^{-k+1}}{2^{-2}}, S^{(1)}, \frac{2^{-k+i-2}}{2^{-i-1}}, S^{(2)},
  \frac{\circ}{2^{-k+1}}, \frac{\circ}{2^{-k+1}}, \frac{\circ}{2^{-m}}, S^{(3)},
  \frac{2^{-k+j-2}}{2^{-j-1}}, S^{(4)} ,\right\}$, or\\
  \item $S(f)=\displaystyle \left\{ S^{(1)}, \frac{2^{-k+i-2}}{2^{-i-1}}, S^{(2)},
  \frac{\circ}{2^{-k+1}}, \frac{\circ}{2^{-k+1}}, \frac{\circ}{2^{-m}}, S^{(3)},
  \frac{2^{-k+j-2}}{2^{-j-1}}, S^{(4)}, \frac{2^{-k+1}}{2^{-2}},\frac{2^{-k+1}}{2^{-1}}\right\}$, or\\
  \item $S(f)=\displaystyle \left\{\frac{2^{-k+1}}{2^{-1-\delta}}, S^{(1)}, \frac{2^{-k+i-2}}{2^{-i-1}}, S^{(2)},
  \frac{\circ}{2^{-k+1}}, \frac{\circ}{2^{-k+1}}, \frac{\circ}{2^{-m}}, S^{(3)},
  \frac{2^{-k+j-2}}{2^{-j-1}}, S^{(4)}, \frac{2^{-k+1}}{2^{-2+\delta}}\right\}$,
  \end{enumerate}
  where $\delta \in \{0,1\}$ and $S^{(q)}$ are subsequences, for $1 \leq q \leq 4$.
  Then, in each one of the above cases we have $r_{i}=-1=-s_{k-1-i}$ and $r_j=1=-s_{k-1-j}$.

  Now we will find a cyclic permutation $\sigma$ such that $s_0$, $s_1$
  and $\sigma(1)$ satisfy the equation (\ref{eq2}) of theorem \ref{mainThm}.
  Let $p$ be the position of the mark on $\Delta_y(f)$. Define $\sigma(1)$
  as the integer satisfying $\sigma(1)+p-1 \equiv 1$ (mod $k$).
  Let $l$ be the position of the first $2^{-k+1}$ on $\Delta_x(f)$. Since
  $S(f)$ is $k-$good we have four possible cases:
  \begin{enumerate}
    \item $S(f)=\displaystyle \left\{ S^{(1)},\frac{\circ 2^{-1}}{2^{-k+1}}, \frac{2^{-2}}{2^{-k+1}}, S^{(2)}\right\}$,
    so that $s_0=s_1=-1$ and $l \equiv p$ (mod $k$).
    \item $S(f)=\displaystyle \left\{ S^{(1)},\frac{2^{-1}}{2^{-k+1}}, \frac{\circ 2^{-2}}{2^{-k+1}}, S^{(2)}\right\}$,
    so that $s_0=1, s_1=-1$ and $l \equiv p-1$ (mod $k$).
    \item $S(f)=\displaystyle \left\{ S^{(1)},\frac{2^{-2}}{2^{-k+1}}, \frac{\circ 2^{-1}}{2^{-k+1}}, S^{(2)}\right\}$,
    so that $s_0=-1, s_1=1$ and $l \equiv p-1$ (mod $k$).
    \item $S(f)=\displaystyle \left\{ S^{(1)},\frac{2^{-2}}{2^{-k+1}}, \frac{2^{-1}}{2^{-k+1}} \frac{\circ}{2^{-m}}, S^{(2)}\right\}$,
    so that $s_0=s_1=1$ and $l \equiv p-2$ (mod $k$).
  \end{enumerate}
  which correspond exactly with equation (\ref{eq2}).

  Finally, for the `only if' part, it suffices to remark that
  given a $k-$thin sequence $\Delta_x(f)$, there are exactly
  four $k-$good sequences $S(f)$. This coincides with the number
  of thin marked trees satisfying the equations of theorem \ref{mainThm}
  that could be a target tree for a given thin source tree.
\end{proof}

\bibliographystyle{amsplain}
\bibliography{referencies}

\end{document}